\documentclass[a4j,12pt,final]{article}
\usepackage{amsmath}
\usepackage{amssymb}
\usepackage{amsthm}
\usepackage{amscd}
\usepackage{mathrsfs}
\newtheorem{theorem}{Theorem}[section]
\newtheorem{lemma}{Lemma}[section]
\theoremstyle{definition}
\newtheorem{definition}{Definition}[section]
\newtheorem{example}{Example}[section]
\theoremstyle{remark}
\newtheorem{remark}{Remark}[section]
\theoremstyle{corollary}

\theoremstyle{errata}
\newtheorem{errata}{Errata}

\begin{document}

\title{The equivariant de Rham complex on a simplicial $G_*$-manifold}
\author{Naoya Suzuki}
\date{}
\maketitle


\begin{abstract}
We show that when a simplicial Lie group acts on a simplicial manifold $\{X_*\}$, we can construct a bisimplicial manifold and the de Rham complex on it.
This complex is quasi-isomorphic to the equivariant simplicial de Rham complex on $\{X_*\}$ and its cohomology group is isomorphic to the cohomology group of the 
fat realization of the bisimplicial manifold. We also exhibit a cocycle in the equivariant simplicial de Rham complex.
\end{abstract}

\section{Introduction}
Simplicial manifold is a sequence of manifolds together
with face and degeneracy operators satisfying some relations. There is a well-known way to construct the de Rham complex on a simplicial manifold  (see \cite{Bot2}\cite{Dup2}\cite{Mos}, for instance).
In \cite{Mei}, Meinrenken introduced the equivariant version of the de Rham complex on a simplicial manifold. That is a double complex 
whose components are equivariant differential forms which is called the Cartan model \cite{Ber}. This complex is a generalization of Weinstein's one in \cite{Wein}. In this paper, we show that when a simplicial Lie group acts on a simplicial manifold $\{X_*\}$, we can construct a bisimplicial manifold and explain that the de Rham complex on it is quasi-isomorphic to the equivariant de Rham complex on $\{X_*\}$. We explain also that its cohomology group is isomorphic  to the cohomology group of the fat realization of the bisimplicial manifold. At the last section, we exhibit a cocycle in the equivariant de Rham complex on a simplicial manifold $ NSO(4)$.
\\

\section{Review of the simplicial de Rham complex}

\subsection{Simplicial manifold}

\begin{definition}[\cite{Seg}]
Simplicial manifold is a sequence of manifolds
$X=\{ {X}_{q} \} , (q=0,1,2 \cdots )$ together with face operaters ${\varepsilon}_{i} : {X}_{q} \rightarrow {X}_{q-1}~ (i=0,1,2 \cdots q)$
and degeneracy operater ${\eta}_{i} : {X}_{q} \rightarrow {X}_{q+1}~ (i=0,1,2 \cdots q)$ which are all smooth maps and satisfy the following identities:

$${\varepsilon}_{i}{\varepsilon}_{j}={\varepsilon}_{j-1}{\varepsilon}_{i}  \qquad i<j$$
$${\eta}_{i}{\eta}_{j}={\eta}_{j+1}{\eta}_{i}  \qquad i \leq j$$
$$
{\varepsilon}_{i}{\eta}_{j}=\begin{cases}
{\eta}_{j-1}{\varepsilon}_{i}  &  i<j \\
id  &  i=j , \enspace i=j+1 \\
{\eta}_{j}{\varepsilon}_{i-1}  &  i>j+1.
\end{cases}
$$
\end{definition}

Simplicial Lie group $\{ G_* \}$ is a simplicial manifold such that all $G_n$ are Lie groups and all face and degeneracy operators are group homomorphisms.\\

For any Lie group $G$, we have simplicial manifolds $NG$, $PG$ and simplicial $G$-bundle  $\gamma : PG \rightarrow NG$
as follows:\\
\par
$NG(q)  = \overbrace{G \times \cdots \times G }^{q-times}  \ni (g_1 , \cdots , g_q ) :$  \\
face operators \enspace ${\varepsilon}_{i} : NG(q) \rightarrow NG(q-1)  $
$$
{\varepsilon}_{i}(g_1 , \cdots , g_q )=\begin{cases}
(g_2 , \cdots , g_q )  &  i=0 \\
(g_1 , \cdots ,g_i g_{i+1} , \cdots , g_q )  &  i=1 , \cdots , q-1 \\
(g_1 , \cdots , g_{q-1} )  &  i=q
\end{cases}
$$

\par
\medskip
$PG (q) = \overbrace{ G \times \cdots \times G }^{q+1 - times} \ni (\bar{g}_1 , \cdots , \bar{g}_{q+1} ) :$ \\
face operators \enspace $ \bar{\varepsilon}_{i} : PG(q) \rightarrow PG(q-1)  $ 
$$ \bar{{\varepsilon}} _{i} (\bar{g}_1 , \cdots , \bar{g}_{q+1} ) = (\bar{g}_1 , \cdots , \bar{g}_{i} , \bar{g}_{i+2}, \cdots , \bar{g}_{q+1})  \qquad i=0,1, \cdots ,q $$

\par
\medskip
Degeneracy operators are also defined but we do not need them here.\\

We define $\gamma : PG \rightarrow NG $ as $ \gamma (\bar{g}_1 , \cdots , \bar{g}_{q+1} ) = (\bar{g}_1 {\bar{g}_2}^{-1} , \cdots , \bar{g}_{q} {\bar{g}_{q+1}}^{-1} )$.\\

For any simplicial manifold $\{ X_* \}$, we can associate a topological space $\parallel X_* \parallel $ 
called the fat realization defined as follows:
$$  \parallel X_* \parallel  := \coprod _{n}  {\Delta}^{n} \times X_n / \enspace ( {\varepsilon}^{i} t , x) \sim (  t , {\varepsilon}_{i} x).$$
Here ${\Delta}^{n}$ is the standard $n$-simplex and ${\varepsilon}^{i}$ is a face map of it.
It is well-known that 
$\parallel \gamma \parallel : \parallel PG \parallel \rightarrow \parallel NG \parallel$ is the universal bundle $EG \rightarrow BG$  (see \cite{Dup2} 
\cite{Mos} \cite{Seg}, for instance). \\

\subsection{The double complex on a simplicial manifold}

\begin{definition}
For any simplicial manifold $ \{ X_* \}$ with face operators $\{ {\varepsilon}_* \}$, we have a double complex ${\Omega}^{p,q} (X_*) := {\Omega}^{q} (X_p) $ with derivatives defined as follows:
$$ d' := \sum _{i=0} ^{p+1} (-1)^{i} {\varepsilon}_{i} ^{*}  , \qquad  d'' := (-1)^{p} \times {\rm the \enspace exterior \enspace differential \enspace on \enspace }{ \Omega ^*(X_p) } .$$

\end{definition}
\bigskip

For any simplicial manifold the following holds.

\begin{theorem}[\cite{Bot2} \cite{Dup2} \cite{Mos}]
 There exist a ring isomorphism

$$ H^*({\Omega}^{*} (X_*))  \cong  H^{*} (\parallel X_* \parallel).  $$

 Here ${\Omega}^{*} (X_*)$  means the total complex.
\end{theorem} 
$\hspace{30em} \Box $

\section{Simplicial $G_*$-manifold}

Let $\{ X_* \}$ be a simplicial manifold and $\{ G_* \}$ be a simplicial Lie group which acts on $\{ X_* \}$ by left, i.e. 
 $G_n$ acts on $X_n$ by left and this action is commutative with face and degeneracy operators of $\{X_*\}$. We call $\{ X_* \}$ a simplicial $G_*$-manifold.

A bisimplicial manifold is a sequence of manifolds with horizontal and vertical face and degeneracy operators which commute with each other.

Given a simplicial $G_*$-manifold $\{ X_* \}$, we can construct a bisimplicial manifold $\{ X_* \rtimes NG_*(*) \}$ in the following way:\\

$X_p \rtimes NG_p(q):=X_p \times \overbrace{G_p \times \cdots \times G_p}^{q-{\rm times}}.$\\

Horizontal face operators \enspace ${\varepsilon}_{i}^{Ho} : X_p \rtimes NG_p(q) \rightarrow X_{p-1}  \rtimes NG_{p-1}(q) $ are the same as the face operators of $X_p$ and $G_p$.
Vertical face operators \enspace ${\varepsilon}_{i}^{Ve} : X_p \rtimes NG_p(q) \rightarrow X_p  \rtimes NG_p(q-1) $ are
$$
{\varepsilon}_{i}^{Ve}(x, g_1 , \cdots , g_q )=\begin{cases}
(x, g_2 , \cdots , g_q )  &  i=0 \\
(x, g_1 , \cdots ,g_i g_{i+1} , \cdots , g_q )  &  i=1 , \cdots , q-1 \\
(g_{q}x, g_1 , \cdots , g_{q-1} )  &  i=q.
\end{cases}
$$

\begin{example}
Suppose $G_n=H$ is a compact subgroup of $G$ and $H$ acts on $NG(n)$ as follows:
$$h \cdot (g_1, g_2, \cdots, g_n)  = (hg_1h^{-1}, hg_2h^{-1}, \cdots, hg_nh^{-1}).$$
Then $X_n=NG(n)$ is a simplicial $H$-manifold and 
$\parallel NG(*) \rtimes NH(*) \parallel$ is $B(G \rtimes H)$ (\cite{Suz}).

\end{example}

\begin{example}
$PG(n)$ acts on $PG(n)$ itself by left as follows:
$$(\bar{k}_1, \cdots, \bar{k}_{n+1}) \cdot (\bar{g}_1, \cdots, \bar{g}_{n+1}) = (\bar{k}_1\bar{g}_1\bar{k}^{-1}_1, \cdots, \bar{k}_{n+1}\bar{g}_{n+1}\bar{k}^{-1}_{n+1}).$$
So $PG(*)$ is a simplicial $PG(*)$-manifold (\cite{Mei}). If $G$ is compact, $\parallel PG(*) \rtimes N(PG(*))(*) \parallel$ is a fat realization
of a simplicial space $PG(n) \times _{PG(n)} EPG(n)$.

\end{example}

\begin{example}
If the action of $\{ G_* \}$ on $\{ X_* \}$ is trivial, $\parallel X_* \rtimes NG_*(*) \parallel$ is $\parallel X_* \parallel \times \parallel BG_* \parallel$.
\end{example}

\begin{example}
Let $\Gamma_1 \rightrightarrows \Gamma_0$ be a $G$-groupoid, i.e. both $\Gamma_1$ and $\Gamma_0$ are 
$G$-manifolds and all structure maps are $G$-equivariant.
We define a simplicial manifold $N\Gamma$ as follows:
$$N\Gamma(p)  := \{ (x_1, \cdots , x_p) \in \overbrace{\Gamma_1 \times \cdots \times \Gamma_1 }^{p-times}  \enspace | \enspace t(x_{j})=s(x_{j+1}) \enspace j=1 , \cdots , p-1 \}$$  
face operators \enspace ${\varepsilon}_{i} : N\Gamma(p) \rightarrow N\Gamma(p-1)  $
$$
{\varepsilon}_{i}(x_1 , \cdots , x_p )=\begin{cases}
(x_2 , \cdots , x_p )  &  i=0 \\
(x_1 , \cdots ,m(x_i ,x_{i+1}) , \cdots , x_p )  &  i=1 , \cdots , p-1 \\
(x_1 , \cdots , x_{p-1} )  &  i=p .
\end{cases}
$$
Here $s,t,m$ mean the source and target maps, and the multiplication (\cite{Sti}). Then $N\Gamma(*)$ is a simplicial $G$-manifold.

\end{example}
\section{The equivariant simplicial de Rham complex}
\subsection{The triple complex}

\begin{definition}
For a bisimplicial manifold $\{ X_{*,*} \}$, we can construct a triple complex on it in the following way:

$${\Omega}^{p,q,r} (X_{*,*}) := {\Omega}^{r} (X_{p,q}) $$

Derivatives are:
$$ d' := \sum _{i=0} ^{p+1} (-1)^{i} ({{\varepsilon}^{Ho} _{i}}) ^{*}  , \qquad  d'' := \sum _{i=0} ^{q+1} (-1)^{i} ({{\varepsilon}^{Ve} _{i}}) ^{*} \times (-1)^{p} $$
$$ d''' :=  (-1)^{p+q} \times {\rm the \enspace exterior \enspace differential \enspace on \enspace }{ \Omega ^*(X_{p,q}) }.$$

\end{definition}

Repeating the same argument in \cite{Suz}, we obtain the following theorem.
\begin{theorem}
There exists an isomorphism 

$$ H({\Omega}^{*} (X_* \rtimes NG_*(*))) \cong  H^{*} (\parallel X_* \rtimes NG_*(*) \parallel).$$

 Here ${\Omega}^{*} (X_* \rtimes NG_*(*))$  means the total complex.
\end{theorem} 
$\hspace{30em} \Box $
\subsection{The equivariant simplicial de Rham complex}

When a compact Lie group $G$ acts on a manifold $M$, there is the complex of equivariant differential forms 
${\Omega}_G ^{*} (M) := ( {\Omega} ^{*} (M) \otimes S(\mathcal{G}^*))^G$ with the differential $d_G$ defined by $(d_G \alpha)(X):=(d-i_{X_M})(\alpha (X))$ (\cite{Ber} \cite{Car}).  
Here $\mathcal{G}$ is the
Lie algebra of $G$,
$S(\mathcal{G}^*)$
is the algebra of polynomial functions on $\mathcal{G}$, $\alpha \in \Omega^*_G(M), X \in \mathcal{G}$ and $X_M$
denote the vector field on $M$ generated by $X$. This is
called the Cartan Model.
We can define the double complex ${\Omega}^{*} _{G_*} (X_*)$ in the same way as in Definition 2.2. 
This double complex is originally introduced by Meinrenken in \cite{Mei}.

Again, repeating the same argument in \cite{Suz}, we obtain the following theorem.
\begin{theorem}
 If every $G_n$ is compact, there exists an isomorphism

$$ H({\Omega}_{G_*} ^{*} (X_*))  \cong H({\Omega}^{*} (X_* \rtimes NG_*(*))).$$

 Here ${\Omega}_{G_*} ^{*} (X_*)$  means the total complex.
\end{theorem} 
$\hspace{30em} \Box $

\begin{remark}
In the case that $G_n$ is not compact, we need to use ``the Getzler model" of the equivariant cohomology in \cite{Get}.

\end{remark}

\subsection{Cocycle in the equivariant simplicial de Rham complex}
In this section we take $G=SO(4)$ and construct a cocycle in $ \Omega ^{4}_{SO(4)} (NSO(4)) $, whose cohomology is 
isomorphic to $H^*(B(SO(4) \rtimes SO(4)))$.

Recall that there is a cocycle in $ \Omega ^{4} (NSO(4)) $ described in the following way.
\begin{theorem}[\cite{Suz2}]
The cocycle which represents the Euler class of $ESO(4) \rightarrow BSO(4)$ in $ \Omega ^{4} (NSO(4)) $ is the sum of the following
$E_{1,3}$ and $E_{2,2}$:
$$
\begin{CD}
0 \\
@AA{-d}A \\
E_{1,3} \in {\Omega}^{3} (SO(4) )@>{d'}>>{\Omega}^{3} (SO(4) \times SO(4))\\
@.@AA{d}A\\
@.E_{2,2} \in {\Omega}^{2} (SO(4) \times SO(4))@>{d'}>> 0
\end{CD}
$$
$$E_{1,3} =  \frac{1}{192 \pi ^2} \sum_{\tau \in \mathfrak{S} _{4}}   {\rm sgn} (\tau)  \bigl((h^{-1}dh)_{\tau (1) \tau(2)}(h^{-1}dh)^2 _{\tau (3) \tau(4)}~~~~~~~~~~~~~~~~~~~~~~~~~~~~$$
$$~~~~~~~~~~~~~~~~~~~~~~~~~~~~~~~~~~~~~~~~~~~~~~~~~+(h^{-1}dh) _{\tau (3) \tau(4)}(h^{-1}dh)^2 _{\tau (1) \tau(2)} \bigl)$$
$$ E_{2,2} =  \frac{-1}{64 \pi ^2} \sum_{\tau \in \mathfrak{S} _{4}}   {\rm sgn} (\tau) \bigl((h_1 ^{-1}dh_1)_{\tau (1) \tau(2)}(dh_2 h_2^{-1}) _{\tau (3) \tau(4)}~~~~~~~~~~~~~~~~~~~~~~~~~~~~$$
$$~~~~~~~~~~~~~~~~~~~~~~~~~~~~~~~~~~~~~~~~~~~~~~+(h_1^{-1}dh_1) _{\tau (3) \tau(4)}(dh_2h_2^{-1}) _{\tau (1) \tau(2)} \bigl).$$

\end{theorem}
$\hspace{30em} \Box $
\begin{errata}
{In} \cite{Suz2}, there are some mistakes. Some numbers of propositions and theorems are wrong. For example, ``Proposition 3.1"
in P.38 should be modified as ``Proposition 2.1". The cocycle in Theorem 2.2 should be written as above. 
Also, $\displaystyle [\frac{ \partial ^2}{\partial y_1  \partial y_2}b(\gamma_1,\gamma_2)]_{y_i=0}$ and $\alpha(\xi_1,\xi_2)$ should be written as follows.
$$ [\frac{ \partial ^2}{\partial y_1  \partial y_2}b(\gamma_1,\gamma_2)]_{y_i=0}= \frac{-1}{64 \pi ^2} \sum_{\tau \in \mathfrak{S} _{4}}   {\rm sgn} (\tau) \int^1_0 \Bigl(\left( \frac{\partial\xi_1(\theta)}{\partial\theta} \right)_{\tau(1)\tau(2)} \xi_2(\theta)_{\tau(3)\tau(4)}$$
$$~~~~~~~~~~~~~~~+\left( \frac{\partial\xi_1(\theta)}{\partial\theta} \right)_{\tau(3)\tau(4)} \xi_2(\theta)_{\tau(1)\tau(2)}\Bigl)d\theta.$$

$$\alpha(\xi_1,\xi_2):=\frac{-1}{64 \pi ^2} \sum_{\tau \in \mathfrak{S} _{4}} \biggl( {\rm sgn} (\tau) \cdot ~~~~~~~~~~~~~~~~~~~~~~~~~~~~~~~~~~~~~~~~~~~~~ $$
$$ \int^1_0 \Bigl( \left( \frac{\partial\xi_1(\theta)}{\partial\theta} \right)_{\tau(1)\tau(2)} \xi_2(\theta)_{\tau(3)\tau(4)} + \left( \frac{\partial\xi_1(\theta)}{\partial\theta} \right)_{\tau(3)\tau(4)} \xi_2(\theta)_{\tau(1)\tau(2)} $$
$$  - \left( \frac{\partial\xi_2(\theta)}{\partial\theta} \right)_{\tau(1)\tau(2)} \xi_1(\theta)_{\tau(3)\tau(4)} -\left( \frac{\partial\xi_2(\theta)}{\partial\theta} \right)_{\tau(3)\tau(4)} \xi_1(\theta)_{\tau(1)\tau(2)} \Bigl) d\theta \biggl).$$
\end{errata}
$\hspace{30em} \Box $

Now following Jeffrey and Weinstein's idea, we construct a cocycle in $ \Omega ^{4}_{SO(4)} (NSO(4)) $.

We take a cochain $\mu \in ( {\Omega} ^{1} (G) \otimes \mathcal{G}^*)^G$ as follows:
$$\mu(X) =  \frac{-1}{64 \pi ^2} \sum_{\tau \in \mathfrak{S} _{4}}   {\rm sgn} (\tau)  \bigl((X) _{\tau (1) \tau(2)}(h^{-1}dh) _{\tau (3) \tau(4)}+(X) _{\tau (3) \tau(4)}(h^{-1}dh)_{\tau (1) \tau(2)}\bigl)$$
$$-  \frac{1}{64 \pi ^2} \sum_{\tau \in \mathfrak{S} _{4}}   {\rm sgn} (\tau)  \bigl((X) _{\tau (1) \tau(2)}(dhh^{-1}) _{\tau (3) \tau(4)}+(X) _{\tau (3) \tau(4)}(dhh^{-1})_{\tau (1) \tau(2)}\bigl).$$
Here $X \in \mathcal{G}=\mathfrak{so}(4)$.
\begin{lemma}
$i_{X_G}E_{1,3}=d\mu (X) $
\end{lemma}
\begin{proof}
Since $i_X({g^{-1}dg})=i_{\bar{X}}({dgg^{-1}})=X$, the following equation holds.
$$i_{X_G}E_{1,3}=i_{X-\bar{X}}E_{1,3}=i_{X}E_{1,3}-i_{\bar{X}}E_{1,3}~~~~~~~~~~~~~~~~~~~~~~~~~~~~~~~~~~~~~~~~~~~~~$$
$$ =  \frac{1}{64 \pi ^2} \sum_{\tau \in \mathfrak{S} _{4}}   {\rm sgn} (\tau)  \bigl((X)_{\tau (1) \tau(2)}(h^{-1}dh)^2 _{\tau (3) \tau(4)}+(X) _{\tau (3) \tau(4)}(h^{-1}dh)^2 _{\tau (1) \tau(2)} \bigl)$$
$$ -  \frac{1}{64 \pi ^2} \sum_{\tau \in \mathfrak{S} _{4}}   {\rm sgn} (\tau)  \bigl((X)_{\tau (1) \tau(2)}(dhh^{-1})^2 _{\tau (3) \tau(4)}+(X) _{\tau (3) \tau(4)}(dhh^{-1})^2 _{\tau (1) \tau(2)} \bigl)$$
$$=d\mu (X).~~~~~~~~~~~~~~~~~~~~~~~~~~~~~~~~~~~~~~~~~~~~~~~~~~~~~~~~~~~~~~~~~~~~~~~~~~~~~~~~$$
\end{proof}

\begin{lemma}
$i_{X_{G \times G}}E_{2,2}=({\varepsilon}^*_{0}-{\varepsilon}^*_{1}+{\varepsilon}^*_{2})\mu(X)$
\end{lemma}
\begin{proof}
$({\varepsilon}^*_{0}-{\varepsilon}^*_{1}+{\varepsilon}^*_{2})\mu(X)$
$$ = \frac{-1}{64 \pi ^2} \sum_{\tau \in \mathfrak{S} _{4}}   {\rm sgn} (\tau)  \bigl((X) _{\tau (1) \tau(2)}(h_2^{-1}dh_2) _{\tau (3) \tau(4)}+(X) _{\tau (3) \tau(4)}(h_2^{-1}dh_2)_{\tau (1) \tau(2)}\bigl)$$
$$-  \frac{1}{64 \pi ^2} \sum_{\tau \in \mathfrak{S} _{4}}   {\rm sgn} (\tau)  \bigl((X) _{\tau (1) \tau(2)}(dh_2h_2^{-1}) _{\tau (3) \tau(4)}+(X) _{\tau (3) \tau(4)}(dh_2h_2^{-1})_{\tau (1) \tau(2)}\bigl)$$

$$ + \frac{1}{64 \pi ^2} \sum_{\tau \in \mathfrak{S} _{4}}   {\rm sgn} (\tau)  \bigl((X) _{\tau (1) \tau(2)}(h_2^{-1}h_1 ^{-1} dh_1h_2+h_2^{-1}dh_2) _{\tau (3) \tau(4)}~~~~~~~~~~~~~~~~~$$
$$~~~~~~~~~~~~~~~~~~~~~~~~~~~~~~~~~~~~~~+(X) _{\tau (3) \tau(4)}(h_2^{-1}h_1 ^{-1} dh_1h_2+h_2^{-1}dh_2)_{\tau (1) \tau(2)}\bigl)$$
$$+  \frac{1}{64 \pi ^2} \sum_{\tau \in \mathfrak{S} _{4}}   {\rm sgn} (\tau)  \bigl((X) _{\tau (1) \tau(2)}(dh_1 h_1^{-1}+h_1dh_2h_2 ^{-1}h_1^{-1}) _{\tau (3) \tau(4)}~~~~~~~~~~~~~~~~~$$
$$~~~~~~~~~~~~~~~~~~~~~~~~~~~~~~~~~~~~~~+(X) _{\tau (3) \tau(4)}(dh_1 h_1^{-1}+h_1dh_2h_2 ^{-1}h_1^{-1})_{\tau (1) \tau(2)}\bigl)$$

$$ - \frac{1}{64 \pi ^2} \sum_{\tau \in \mathfrak{S} _{4}}   {\rm sgn} (\tau)  \bigl((X) _{\tau (1) \tau(2)}(h_1^{-1}dh_1) _{\tau (3) \tau(4)}+(X) _{\tau (3) \tau(4)}(h_1^{-1}dh_1)_{\tau (1) \tau(2)}\bigl)$$
$$-  \frac{1}{64 \pi ^2} \sum_{\tau \in \mathfrak{S} _{4}}   {\rm sgn} (\tau)  \bigl((X) _{\tau (1) \tau(2)}(dh_1h_1^{-1}) _{\tau (3) \tau(4)}+(X) _{\tau (3) \tau(4)}(dh_1h_1^{-1})_{\tau (1) \tau(2)}\bigl).$$
$$=  \frac{-1}{64 \pi ^2} \sum_{\tau \in \mathfrak{S} _{4}}   {\rm sgn} (\tau)  \bigl((X) _{\tau (1) \tau(2)}(dh_2h_2^{-1}) _{\tau (3) \tau(4)}+(X) _{\tau (3) \tau(4)}(dh_2h_2^{-1})_{\tau (1) \tau(2)}\bigl)$$
$$ + \frac{1}{64 \pi ^2} \sum_{\tau \in \mathfrak{S} _{4}}   {\rm sgn} (\tau)  \bigl((X) _{\tau (1) \tau(2)}(h_2^{-1}h_1 ^{-1} dh_1h_2) _{\tau (3) \tau(4)}~~~~~~~~~~~~~~~~~$$
$$~~~~~~~~~~~~~~~~~~~~~~~~~~~~~~~~~~~~~~+(X) _{\tau (3) \tau(4)}(h_2^{-1}h_1 ^{-1} dh_1h_2)_{\tau (1) \tau(2)}\bigl)$$
$$+  \frac{1}{64 \pi ^2} \sum_{\tau \in \mathfrak{S} _{4}}   {\rm sgn} (\tau)  \bigl((X) _{\tau (1) \tau(2)}(h_1dh_2h_2 ^{-1}h_1^{-1}) _{\tau (3) \tau(4)}~~~~~~~~~~~~~~~~~$$
$$~~~~~~~~~~~~~~~~~~~~~~~~~~~~~~~~~~~~~~+(X) _{\tau (3) \tau(4)}(h_1dh_2h_2 ^{-1}h_1^{-1})_{\tau (1) \tau(2)}\bigl)$$
$$ - \frac{1}{64 \pi ^2} \sum_{\tau \in \mathfrak{S} _{4}}   {\rm sgn} (\tau)  \bigl((X) _{\tau (1) \tau(2)}(h_1^{-1}dh_1) _{\tau (3) \tau(4)}+(X) _{\tau (3) \tau(4)}(h_1^{-1}dh_1)_{\tau (1) \tau(2)}\bigl)$$
$$= i_{X_1}E_{2,2} + i_{X_2}E_{2,2}-i_{\bar{X}_1}E_{2,2}-i_{\bar{X}_2}E_{2,2}= i_{X_1-\bar{X}_1+X_2-\bar{X}_2}E_{2,2}=i_{X_{G \times G}}E_{2,2}.$$

\end{proof}

\begin{lemma}
$-i_{X_G}\mu(X)=0$
\end{lemma}
\begin{proof}
$-i_{X_G}\mu(X)=-i_{X-\bar{X}}\mu(X)$
$$  =\frac{1}{64 \pi ^2} \sum_{\tau \in \mathfrak{S} _{4}}   {\rm sgn} (\tau)  \bigl((X) _{\tau (1) \tau(2)}(X) _{\tau (3) \tau(4)}+(X) _{\tau (3) \tau(4)}(X)_{\tau (1) \tau(2)}\bigl)$$
$$+  \frac{1}{64 \pi ^2} \sum_{\tau \in \mathfrak{S} _{4}}   {\rm sgn} (\tau)  \bigl((h^{-1}Xh) _{\tau (1) \tau(2)}(X) _{\tau (3) \tau(4)}+(h^{-1}Xh) _{\tau (3) \tau(4)}(X)_{\tau (1) \tau(2)}\bigl)$$
$$ - \frac{1}{64 \pi ^2} \sum_{\tau \in \mathfrak{S} _{4}}   {\rm sgn} (\tau)  \bigl((hXh^{-1}) _{\tau (1) \tau(2)}(X) _{\tau (3) \tau(4)}+(hXh^{-1}) _{\tau (3) \tau(4)}(X)_{\tau (1) \tau(2)}\bigl)$$
$$-  \frac{1}{64 \pi ^2} \sum_{\tau \in \mathfrak{S} _{4}}   {\rm sgn} (\tau)  \bigl((X) _{\tau (1) \tau(2)}(X) _{\tau (3) \tau(4)}+(X) _{\tau (3) \tau(4)}(X)_{\tau (1) \tau(2)}\bigl)=0.$$
\end{proof}

As a result, we obtain the following theorem.
\begin{theorem}
$E_{1,3}+E_{2,2}+\mu$ is a cocycle in $ \Omega ^{4}_{SO(4)} (NSO(4)) $.
\end{theorem}

National Institute of the Technology, Akita College, 1-1, Iijima Bunkyo-cho, Akita-shi, Akita-ken, Japan. \\
e-mail: nysuzuki@akita-nct.ac.jp
\end{document}